\documentclass[10pt]{article}
\usepackage{amsmath, amsthm, amssymb, amsfonts, mathrsfs, mathtools, upgreek}
\usepackage{graphicx}
\usepackage{makeidx}
\usepackage{tikz-network}
\usepackage[left=2cm,right=2cm,top=2cm,bottom=2cm]{geometry}
\usepackage{caption}
\usepackage{subcaption}
\usepackage[section]{placeins}
\usepackage{enumitem}
\usepackage{tikz}
\usepackage{stmaryrd}
\usepackage{hyperref}
\usepackage{upquote}

\usepackage{hyperref}
\hypersetup{colorlinks=true, citecolor={blue}, linkcolor=blue, filecolor=magenta, urlcolor=cyan}

\newtheorem{theorem}{Theorem}[section]
\newtheorem{prop}[theorem]{Proposition}
\newtheorem{lema}[theorem]{Lemma}
\newtheorem{corollary}{Corollary}[theorem]
\theoremstyle{definition}
\newtheorem{defi}{Definition}

\numberwithin{equation}{section}
\newcommand{\ip}[2]{\left\langle #1, #2 \right\rangle}
\newcommand{\divides}{\mid}

\newcommand{\real}{\operatorname{Re}}

\newcommand{\iu}{\mathbf{i}}
\newcommand{\spec}{\operatorname{Spec}}
\newcommand{\dist}{\operatorname{dist}}
\newcommand{\lcm}{\operatorname{lcm}}

\def \j {{\mathbf{j}}}

\def \Nl {{\mathbb N}}
\def \Rl {{\mathbb R}}

\def \Ql {{\mathbb Q}}
\def \Cl {{\mathbb C}}
\def \Ca {{\mathbb{C}[\mathcal{A}]}}
\def \Ac {{\mathcal{A}}}

\def \ld {{\lambda}}

\def \eu {{\textbf{e}_u}}
\def \ev {{\textbf{e}_v}}
\def \ew {{\textbf{e}_w}}

\title{Perfect state transfer in Grover walks on association schemes and distance-regular graphs}

\author{ Koushik Bhakta and Bikash Bhattacharjya\\
	Department of Mathematics\\
	Indian Institute of Technology Guwahati, India\\
	b.koushik@iitg.ac.in, b.bikash@iitg.ac.in }

\date{}
\begin{document}
	\maketitle
	
	\vspace{-0.3in}
	
	\begin{center}{\textbf{Abstract}}\end{center}
	\noindent This paper investigates perfect state transfer in Grover walks, a model of discrete-time quantum walks. We establish a necessary and sufficient condition for the occurrence of perfect state transfer on graphs belonging to an association scheme. Our focus includes specific association schemes, namely the Hamming and Johnson schemes. We characterize all graphs on the classes of  Hamming and Johnson schemes that exhibit perfect state transfer. Furthermore, we study perfect state transfer on distance-regular graphs. We provide complete characterizations for exhibiting perfect state transfer on distance-regular graphs of diameter $2$ and diameter $3$, as well as integral distance-regular graphs.

	\vspace*{0.3cm}
	
	\section{Introduction}
	
	Quantum walks are fundamental mechanisms in quantum computing that are widely used for developing quantum algorithms~\cite{ambain} and studying quantum transport phenomena~\cite{kendon}. Quantum walks serve as the quantum analogue of classical random walks. They are broadly classified into two primary forms: continuous-time quantum walks~\cite{state} and discrete-time quantum walks~\cite{zhan3}. A remarkable feature of quantum walks is their ability to achieve perfect state transfer, where a quantum state is transferred from one position to another with probability $1$. Perfect state transfer has been extensively studied in the context of continuous-time quantum walks (see \cite{johnson, association, coutinho, state} and references therein). In recent years, there is an increasing trend in research on discrete-time quantum walks due to their unique applications in graph isomorphism testing~\cite{graphiso}, quantum search algorithms~\cite{search}, and more. There are various models in discrete-time quantum walks (see \cite{godsildct}). One of the most well-studied discrete-time quantum walks on graphs is the Grover walk (see \cite{spec,zhan2}). The Grover walk is a coined quantum walk based on the Grover diffusion operator, which has attracted significant attention due to its intriguing properties, including localization~\cite{localization}, periodicity~\cite{hig1}, and state transfer capabilities~\cite{kubota2}. This paper focuses on the study of perfect state transfer in Grover walks. The Grover walk, also called arc-reversal Grover walk~\cite{zhan1}, is recognized as a special case of bipartite walk~\cite{chen}. In Section~2.1, we provide a detailed definition of Grover walk on graphs. 

	The last few decades have seen growing interest in studying quantum state transfer in Grover walks. Zhan~\cite{zhan1} constructed an infinite family of graphs that exhibit perfect state transfer, which provides a significant step toward understanding this phenomenon. Further contributions include the work of Kubota and Segawa~\cite{kubota1}, who provided a necessary condition for the occurrence of perfect state transfer on graphs and characterized its occurrence on complete multipartite graphs. Kubota and Yoshino~\cite{kubota2} investigated perfect state transfer on circulant graphs with valency up to four. In \cite{bhakta1, bhakta2, bhakta3}, we studied perfect state transfer on some special classes of Cayley graphs. The study on periodicity of Grover walks, which plays a crucial role in perfect state transfer, has been explored in~\cite{ito, mixed3, bipartite,bethetrees, regular, oddperiodic}. Previous research has explored perfect state transfer in specific families of graphs, such as circulant graphs, complete multipartite graphs, and others. However, a systematic study of perfect state transfer on graphs arising from association schemes has been lacking. Yoshie~\cite{distance} studied periodicity on distance-regular graphs. In this paper, we study perfect state transfer on distance-regular graphs, more generally, for graphs that belong to association schemes. Perfect state transfer in the context of continuous-time quantum walks on graphs belonging to association schemes has been studied in~\cite{johnson, association}.

	We now briefly discuss our main results. We refer the reader to the later sections for terminologies and further details. In Theorem~\ref{association2}, we provide a necessary and sufficient condition for the occurrence of perfect state transfer on graphs in an association scheme. In particular, Theorems~\ref{pst_hamming} and \ref{pst_johnson} provide simple conditions for the occurrence of perfect state transfer on graphs in the Hamming and Johnson schemes, respectively. We also fully characterize  perfect state transfer on the classes of Hamming and Johnson schemes in Theorems~\ref{class_hamming} and \ref{class_johnson}. Using Theorem~\ref{association2}, we establish a necessary and sufficient condition for perfect state transfer on distance-regular graphs in Theorem~\ref{pst_distance}. Furthermore, we completely characterize perfect state transfer on distance-regular graphs of diameter 2 and 3 in Theorems~\ref{pst_distance2} and \ref{pst_distance3}, respectively. Finally,  Theorem~\ref{pst_int_distance} provides a complete classification of perfect state transfer on integral distance-regular graphs. Our work provides a connection between algebraic combinatorics and quantum information theory by focusing on association schemes, which offer a rich algebraic structure to study symmetric graphs.

	The paper is organized as follows. Section $2$ provides the necessary preliminaries, including the definition of Grover walks and the concepts of periodicity and perfect state transfer. This section also introduces an overview of association schemes, distance-regular graphs, and some foundational results used in subsequent sections. In Section $3$, we investigate perfect state transfer on association schemes, with a particular focus on the Hamming and Johnson schemes. Section $4$ explores perfect state transfer on distance-regular graphs. In particular, we examine cycles, complete graphs, and distance-regular graphs with diameters $2$ and $3$. We also investigate perfect state transfer on integral distance-regular graphs.
	
	\section{Preliminaries}
	Let $G:=(V(G),E(G))$ be a finite simple graph with vertex set $V(G)$ and edge set $E(G)$. The elements of $E(G)$ are denoted by unordered pairs $uv$, where $u,v\in V(G)$, $u\neq v$, and $uv=vu$. Let $u$ and $v$ be two vertices of a connected graph $G$. The \emph{distance} $\dist_G(u,v)$  is defined as  the length of the shortest path connecting two vertices $u$ and $v$ in $G$. The \emph{diameter} of $G$ is defined as 
	$$\operatorname{diam}(G) = \max_{u, v \in V(G)} \dist_G(u, v).$$
	Two vertices $u$ and $v$ of $G$ are called  \emph{antipodal} if $\dist_G(u,v)=\operatorname{diam}(G)$. The \emph{adjacency matrix} of a graph $G$, denoted $A:=A(G)\in \Cl^{V(G)\times V(G)}$, is defined by	   
	$$A_{uv} = \left\{ \begin{array}{rl}
		1 &\mbox{ if }
		uv\in E(G) \\ 
		0 &\textnormal{ otherwise.}
	\end{array}\right.$$ 
	A graph is called \emph{integral} if the eigenvalues of its adjacency matrix are integers. Let $I$  denote the identity matrix and $J$ the all-ones matrix. The size of the matrices is clear from their context. We denote $\j$ as the all-ones vector.

	\subsection{Grover walks}

	Let $uv$ be an edge of a finite simple  graph $G$. Then the ordered pairs $(u,v)$ and $(v,u)$ are known as the \emph{arcs} of $G$. The set of  symmetric arcs of $G$ is defined as $\mathcal{A}(G)=\{(u, v), (v, u):uv\in E(G) \}$. For an arc $(u,v)$, the vertices $u$ and $v$ are called the \emph{origin} and \emph{terminus}, respectively. Let $a$ be an arc of $G$ such that $a=(u,v)$. We write $o(a)$ and $t(a)$ to denote the origin and terminus of $a$, respectively, that is, $o(a)=u$ and $t(a)=v$.  The inverse arc of $a$, denoted $a^{-1}$, is the arc $(v,u)$.
	
	We now introduce several matrices for the definition of Grover walks on a graph $G$. The \emph{shift matrix} $S:=S(G)\in \mathbb{C}^{\mathcal{A}(G) \times \mathcal{A}(G)}$ of $G$ is defined by $S_{ab}=\delta_{a,b^{-1}}$,  where $\delta_{a,b}$ is the Kronecker delta function. The \emph{boundary matrix} $N:=N(G)\in \mathbb{C}^{V(G)\times \mathcal{A}(G)}$ of $G$ is defined by $N_{ua}=\frac{1}{\sqrt{\deg u}}\delta _{u, t(a)}$, where $\deg u$ is the degree of the vertex $u$. The \emph{coin matrix} $K:=K(G)\in\mathbb{C}^{\mathcal{A}(G) \times \mathcal{A}(G)}$ of $G$ is defined by $K=2N^*N-I$. Using these matrices, the \emph{time evolution matrix} $U:=U(G)\in \mathbb{C}^{\mathcal{A}(G)\times \mathcal{A}(G)}$ of $G$ is defined by $$U=SK.$$ 
	A discrete-time quantum walk on a graph $G$ is governed by a unitary matrix acting on complex-valued functions over the set of  symmetric arcs of $G$. The discrete-time quantum walk defined by the unitary matrix $U$ is known as the \emph{Grover walk}.

	The eigenvalues of $U$ can be derived from a smaller matrix known as the discriminant matrix. The \emph{discriminant} $P:=P(G)\in \mathbb{C}^{V(G)\times V(G)}$ of $G$ is defined by $$P=NSN^*.$$
	See~\cite{kubota2} for more details about the matrices $S$, $N$, $K$, $U$ and $P$. 	For a non-negative integer $k$, the notation $\{a\}^k$  denotes the multi-set $\{a,\ldots,a\}$, where the element $a$ repeats $k$-times. We also denote $\sqrt{-1}$ by $\iu$. The following result follows from \cite[Proposition~1]{spec}
	\begin{theorem}[{\cite[Theorem~4.3]{mixed1}}]\label{evu_grover} 
		Let $\mu_1, \hdots,\mu_n$ be the eigenvalues of the discriminant of a graph $G$. Then the multi-set of eigenvalues of the time evolution matrix is $$\left\{e^{\pm \iu \arccos(\mu_j)}:1\leq j\leq n\right\}\cup \{1\}^{b_1} \cup \{ -1\}^{b_1-1+1_B},$$ where $b_1=|E(G)|-|V(G)|+1$, and $1_B=1$ or $0$ according as $G$ is bipartite or not.	
	\end{theorem}
	Next, we define the periodicity and perfect state transfer in Grover walks.	
	\begin{defi}  
		The Grover walk on a graph $G$ is \emph{periodic} if there exists a positive integer \( \tau \) such that the time evolution matrix $U$ of $G$ satisfies $U^\tau = I$. 
	\end{defi}  
	For simplicity, we say $G$ is periodic to mean that the Grover walk on $G$ is periodic. The smallest positive integer $\tau$ such that $U^\tau =I$ is called the \emph{period} of $G$, and $G$ is called \emph{$\tau$-periodic}.

	Since $U$  is a unitary matrix, it is diagonalizable. Consequently, the periodicity of a graph can be determined from the eigenvalues of its time evolution matrix. For a matrix $M$ associated with a graph $G$, we denote by $\spec_M(G)$ the set of all distinct eigenvalues of the matrix $M$. 
	\begin{lema}[{\cite[Lemma~5.3]{mixed2}}]\label{period}
		A graph $G$ is $\tau$-periodic if and only if $ \eta^\tau  =1$ for every $\eta \in \spec_U(G)$, and for some $\eta\in\spec_U(G)$, $\eta^j\neq 1$ for each $j\in\{1,\hdots,\tau-1\}$.
	\end{lema}
	It follows from Lemma~\ref{period} that the period of a graph can be computed directly from the eigenvalues of $U$.
	\begin{lema}[{\cite[Corollary~2.2.1]{bhakta2}}]\label{periodic}
		Let $\eta_1, \hdots , \eta_t$ be the distinct eigenvalues of the time evolution matrix of a periodic graph $G$. Let $k_1, \hdots, k_t$ be the least positive integers such that $\eta_1 ^{k_1}=1, \hdots,\eta_t^{k_t}=1$. Then the period of $G$ is $\lcm(k_1, \hdots, k_t)$.
	\end{lema}

	For any complex number $z$, let $\real(z)$ denote the real part of $z$, that is, $\real(z)=\frac{z+\bar{z}}{2}$. Define $\Re$ as the set of real parts of roots of unity, that is, $\Re=\left\{\real(z): z\in \Cl~\text{and}~z^n=1 ~\text{for some positive integer $n$}\right\}$. The next result provides a useful tool for analyzing the periodicity of a graph in terms of its discriminant eigenvalues.	
	\begin{lema}[{\cite[Lemma~2.4]{bhakta2}}]\label{m1}
		Let $G$ be a graph with discriminant $P$. Then $G$ is periodic if and only if $\spec_P(G)\subseteq\Re$. 
	\end{lema}
	In~\cite{bhakta2}, the previous result is extended more explicitly.	Let $\Delta=\{a\pm\sqrt{b}:a,b\in \Ql~\text{and}~b ~\text{is not a square}\}$ and $\overline{\Delta}=\Rl\setminus(\Ql\cup\Delta)$. For any subset $F$ of real numbers, define $\spec_P^F(G)=F\cap\spec_P(G)$.
	\begin{theorem}[{\cite[Theorem~4.4]{bhakta2}}]\label{ls}
		Let $G$ be a regular graph with discriminant matrix $P$. Then $G$ is periodic if and only if $$\spec_P^\Ql(G)\subseteq\left\{\pm1,\pm\frac{1}{2},0\right\},~\spec_P^\Delta(G)\subseteq\left\{\pm \frac{\sqrt{3}}{2}, \pm\frac{1}{4}\pm\frac{\sqrt{5}}{4},\pm\frac{1}{\sqrt{2}}\right\}~\text{and}~\spec_P^{\overline{\Delta}}(G)\subseteq \Re.$$
	\end{theorem}

	For $\Phi,\Uppsi\in \Cl^{\mathcal{A}}$, we denote by $\ip{\Phi}{\Uppsi}$ the Euclidean inner product of $\Phi$ and $\Uppsi$. A vector $\Phi\in \Cl^\mathcal{A}$ is called a \emph{state} if $\ip{\Phi}{\Phi}=1$. We say that perfect state transfer occurs from a state $\Phi$ to another state $\Psi$ at time $\tau\in \Nl$ if there exists a unimodular complex number $\gamma$ such that $$U^\tau\Phi=\gamma \Psi.$$ 
	The next lemma provides an equivalent definition of perfect state transfer between two distinct states.
	\begin{lema}[{\cite[Section~4]{pgst1}}]\label{st11}
		Let $G$ be a graph and $U$ be its time evolution matrix. Then perfect state transfer occurs from a state $\Phi$ to another state $\Psi$ at time $\tau$ if and only if  $|\ip{U^\tau \Phi}{\Uppsi}|=1.$ 
	\end{lema}
	We study the transfer of quantum states that are localized at the vertices of a graph.  A state $\Phi\in \Cl^\mathcal{A}$ is said to be \emph{vertex-type} of a graph $G$ if there exists a vertex $u\in V(G)$ such that $\Phi=N^* \eu$, where $\eu$ is the unit vector corresponding to the vertex $u$, defined by $(\eu)_x=\delta_{u,x}$. In this paper, we consider only vertex-type states. We refer the reader to \cite{kubota1} for a visual illustration and motivation of vertex-type states. The quantum state associated with vertex $u$ is defined as $\Phi_u=N^* \eu$.
	\begin{defi}
		A graph exhibits \emph{perfect state transfer} from vertex $u$ to another vertex $v$ at time $\tau\in\Nl$ if there exists a unimodular complex number $\gamma$ such that $U^\tau\Phi_u = \gamma \Phi_v$.
	\end{defi}
	We say that a graph exhibits perfect state transfer if there exist vertices $u$ and $v$ in the graph such that perfect state transfer occurs from $u$ to $v$ at some time $\tau$. There is a remarkable connection between the occurrence of perfect state transfer and the Chebyshev polynomials (see Lemma~\ref{defpst}).
	
	The \emph{Chebyshev polynomial of the first kind}, denoted $T_m(x)$, is the polynomial defined by the  recurrence relation
	$$T_0(x)=1,~ T_1(x)=x~ \text{and} ~T_m(x)=2xT_{m-1}(x)-T_{m-2}(x) ~\text{for}~ m \geq 2.$$ 
	It is well known that 
	\begin{equation}\label{chb}
		T_m(\cos\theta)=\cos(m\theta).
	\end{equation}
	By \eqref{chb}, it follows that $|T_m(x)|\leq 1$ for $|x|\leq 1$. This yields the following lemma.
	\begin{lema}\label{ch}
		Let $\mu\in[-1,1]$, and let $m$ be any positive integer. Then
		\begin{enumerate}[label=(\roman*)]
			\item $|T_m (\mu)|\leq 1$.
			\item $T_m (\mu)= 1$ if and only if $\mu =\cos\frac{s}{m}\pi$ for some even positive integer $s$.
			\item $T_m (\mu)= -1$ if and only if $\mu =\cos\frac{s}{m}\pi$ for some odd positive integer $s$.
		\end{enumerate}
	\end{lema}

	In \cite{kubota1}, Kubota and Segawa established a relationship between the time-evolution matrix and the discriminant matrix via Chebyshev polynomials.
	\begin{lema}[{\cite[Lemma~3.1]{kubota1}}]\label{ch11}
		Let $T_m(x)$ be the Chebyshev polynomial of the first kind.	Let $G$ be a graph with the time evolution matrix $U$ and discriminant $P$. Then $NU^\tau N^* = T_\tau (P)$ for $\tau \in \Nl \cup \{0\}$. 
	\end{lema}
	Note that the discriminant $P$ of a graph is a symmetric matrix. In fact, if a graph is $k$-regular then $P=\frac{1}{k}A$ (see Lemma~\ref{reg}). Thus, it has a spectral decomposition. Let the distinct eigenvalues of $P$ be $\mu_0,\hdots,\mu_d$ such that $\mu_0>\cdots>\mu_d$. For each $j\in\{0,\hdots,d\}$, let $E_j$ denote the orthogonal projection matrix onto the eigenspace corresponding to $\mu_j$. Then the spectral decomposition of $P$ is
	\begin{equation}\label{sd1}
		P=\sum_{j=0}^{d} \mu_j E_j. 
	\end{equation}
	For $j\in\{0,\hdots,d\}$,	the projection matrix $E_j$ satisfy the property
	\begin{equation}\label{sd11}
		E_j^*=E_j,\quad E_iE_j=\delta_{i,j}E_j,\quad \text{and}\quad \sum_{j=0}^{d}E_j=I.
	\end{equation}
	Using \eqref{sd1} and \eqref{sd11}, for any polynomial $h(x)$, we have
	\begin{equation}\label{sd}
		h(P)=\sum_{j=0}^{d} h(\mu_j) E_j.
	\end{equation} 
	We now briefly review association schemes and distance-regular graphs. For a detailed discussion of association schemes, we refer the reader to \cite{algcomb}, and for distance-regular graphs, to \cite{brouwer_drt}. 
	\subsection{Association schemes}
	An \emph{association scheme} with $d$ classes is a set $\mathcal{A}:=\{A_0,\hdots,A_d\}$ of $n\times n$ $(0,1)$-matrices that satisfies the following conditions:
	\begin{enumerate}
		\item[(i)] $A_0=I$,
		\item[(ii)] $\sum_{i=0}^{d}A_i=J$,
		\item[(iii)] $A_i$ is symmetric for $i\in\{0,\hdots,d\}$, and
		\item[(iv)] for all $i$ and $j$, the product $A_iA_j$ is a linear combination of $A_0,\hdots,A_d$.
	\end{enumerate}
	The matrices $A_i$ are called \emph{classes} of the association scheme. The matrix algebra spanned by these classes is known as the \emph{Bose–Mesner algebra}, typically denoted by $\mathbb{C}[\mathcal{A}]$. We say a graph belongs to an association scheme $\Ac$ if its adjacency matrix lies in the Bose–Mesner algebra $\Ca$.
	
	By properties (iii) and (iv), $\mathbb{C}[\mathcal{A}]$ forms a commutative algebra. Moreover, since $J \in \mathbb{C}[\mathcal{A}]$, every matrix in $\mathbb{C}[\mathcal{A}]$ has constant row sum and constant column sum. Thus, a graph in an association scheme must be regular. By the property (ii), the classes $A_i$ are linearly independent, and so $\mathcal{A}$ is a basis of $\Cl[\mathcal{A}]$.  Since the matrices $A_i$ are symmetric and commute with each other, they can be simultaneously diagonalized. Therefore, there exists a set of orthogonal idempotent matrices ${E_0, \ldots, E_d}$ that also forms a basis for $\mathbb{C}[\mathcal{A}]$. This leads to the following theorem.
	\begin{theorem}[{\cite[Section 11.2]{spectraofgraphs}}]\label{sd3}
		Let $\Ac$ be an association scheme with $d$ classes. Then there exists a basis $\{E_0,\hdots,E_d\}$ of $\Ca$ and numbers $p_i(j)$ such that
		\begin{enumerate}
			\item[(i)] $E_0=\frac{1}{n}J$,
			\item[(ii)]$\sum_{j=0}^{d}E_j=I$, $E_iE_j=\delta_{i,j}E_i$, $E_j^*=E_j$, for each $i,j\in\{0,\hdots,d\}$,  and
			\item[(iii)] $A_iE_j=p_i(j)E_j$ for each $i,j\in\{0,\hdots,d\}$.
		\end{enumerate}
	\end{theorem}
	The matrices $E_0,\hdots,E_d$ are called the \emph{principal idempotents} of the association scheme. The numbers $p_i(j)$ for $i,j\in\{0,\hdots,d\}$ are referred to as the \emph{eigenvalues} of the scheme. Let $m_j=\operatorname{rank}(E_j)$ for $j\in\{0,\hdots,d\}$. Then by property (iii) of the previous theorem, $m_j$ is the multiplicity of the eigenvalue $p_i(j)$ of $A_i$, assuming $p_i(j)\neq p_i(k)$ for $j\neq k$. In particular, we have $m_0 = 1$ and $\sum_{j=0}^{d} m_j = n$. See~\cite{spectraofgraphs} for more discussion. Note that any matrix in the Bose-Mesner algebra $\Ca$ with $d$ classes has at most $d+1$ distinct eigenvalues. This leads to the following result.
	
	\begin{lema}\label{sd4}
		Let $G$ be a graph belonging to an association scheme of $d$ classes. If $G$ has $d+1$ distinct adjacency eigenvalues, then the principal idempotents of the scheme are precisely the principal idempotents in the spectral decomposition of the adjacency matrix of $G$.  
	\end{lema}
	Let $G$ be a graph belonging to an association scheme of $d$ classes. Then by Lemma~\ref{sd4}, if $G$ has $d+1$ distinct adjacency eigenvalues, then the idempotents $E_j$ in \eqref{sd1} are the same as those described in Theorem~\ref{sd3}.
	\subsection{Distance-regular graphs}\label{def_distance}
	A connected graph is called \emph{distance-regular} if, for any two vertices $u$ and $v$, the number of vertices that are at distance $i$ from $u$ and distance $j$ from $v$ depends only on $i$, $j$, and the distance between $u$ and $v$.
	
	Let $G$ be a distance-regular graph with diameter $d$. For $i\in\{0,\hdots,d\}$ and $u\in V(G)$, we define $\Gamma_i(u)$ as the set of all vertices in $G$ that are at distance $i$ from $u$. Then there exist non-negative integers $b_0,\hdots,b_{d-1}$ and $c_1,\hdots,c_d$ such that for any two vertices $u$ and $v$ at distance $i$, the following hold:
	\begin{align*}
		b_i&=|\Gamma_{i+1}(u)\cap \Gamma_1(v)|\quad \text{for}~i\in\{0,\hdots,d-1\},~\text{and}\\
		c_i&=|\Gamma_{i-1}(u)\cap \Gamma_1(v)|\quad \text{for}~i\in\{1,\hdots,d\}.
	\end{align*}
	Note that $c_1=1$ and the graph is $b_0$-regular. Hence the number of neighbors of $v$ that are at distance $i$ from $u$ is denoted by $a_i$, and is determined by the relations
	$$a_i+b_i+c_i=b_0~\text{for}~i\in\{1,\hdots,d-1\},~\text{and}~a_d+c_d=b_0.$$
	The list of parameters
	$$\{b_0,\hdots,b_{d-1};c_1,\hdots,c_d\}$$
	is called the \emph{intersection array} of a distance-regular graph. The intersection array satisfies the following well-known necessary conditions.
	\begin{lema}[{\cite[Proposition~4.1.6]{brouwer_drt}}]\label{intarray}
		The intersection array $\{b_0,\hdots,b_{d-1};c_1,\hdots,c_d\}$ of a distance-regular graph with valency $r$ satisfies the following two conditions.
		\begin{enumerate}
			\item[(i)] $r=b_0\geq \cdots\geq b_{d-1}>0$,  and
			\item[(ii)] $1=c_1\leq\cdots\leq c_d\leq b_0.$
		\end{enumerate}
	\end{lema}
	Let $G$ be a distance-regular graph with diameter $d$. For each $i\in\{1,\hdots,d\}$, define the graph $G_i$ with vertex set $V(G)$ and two vertices $u$ and $v$ are adjacent in $G_i$ if their distance in $G$ is exactly $i$, that is, $\dist_G(u,v)=i$. Let $A$ be the adjacency matrix of $G$. Also, let $A_0=I$ and $A_i$ be the adjacency matrix of $G_i$ for each $i\in\{1,\hdots,d\}$.  By the definition of distance-regular graph, one can prove that 
	\begin{equation}\label{recc}
		AA_i=b_{i-1}A_{i-1}+a_iA_i+c_{i+1}A_{i+1}~\text{for}~i\in\{1,\hdots,d-1\}
	\end{equation}
	Using induction and the Equation~\eqref{recc}, one can show that there exist constants $p_{ij}^k$, which depend only on the intersection array, such that
	$$A_iA_j=\sum_{k=0}^{d}p_{ij}^kA_k.$$
	Therefore, if $G$ is distance-regular graph with diameter $d$, the set of matrices $\{A_0,\hdots,A_d\}$ form an association scheme. However, not every association scheme arises from a distance-regular graph. There are specific conditions that characterize when an association scheme corresponds to a distance-regular graph. For more details, we refer the reader to \cite[Theorems 5.6 and 5.16]{algasso}.
	
	A distance-regular graph $G$ of diameter $d$ is said to be \emph{antipodal} if the vertices at distance $d$ from a given vertex are also at distance $d$ from each other. In this case, the graph $G_d$ is a disjoint union of cliques of equal size, and these cliques are referred to as \emph{fibres}.
	\section{Perfect state transfer on association schemes}
	This section investigates perfect state transfer on graphs in association schemes. We first provide a necessary and sufficient condition for perfect state transfer to occur on graphs. The next result provides a connection between the discriminant and the adjacency matrix of a graph. Let $D:=D(G)$ be the diagonal matrix whose diagonal entries are the
	degrees of the vertices of a graph $G$. If $D=\operatorname{diag}(d_1,\hdots,d_n)$, then for any real number $s$, we define $D^s:=\operatorname{diag}(d_1^s,\hdots,d_n^s)$.
	\begin{lema}\label{evc_P}
		Let $P$ and $A$ be the discriminant and adjacency matrices, respectively, of a graph. Then $P=D^{-\frac{1}{2}}AD^{-\frac{1}{2}}$. Further, $1$ is an eigenvalue of $P$, and for any eigenvalue $\mu$ of $P$,  $|\mu|\leq1$.
	\end{lema}
	\begin{proof}
		For two vertices $u$ and $v$ of the graph, we have
		\begin{align*}	
			P_{uv}&=\sum_{e,f\in\mathcal{A}(G)}N_{ue}S_{ef}N_{vf}\\
			&=\mathop{\sum_{e,f\in\mathcal{A}(G)}}_{t(e)=u, t(f)=v} \frac{1}{\sqrt{\deg u}}\frac{1}{\sqrt{\deg v}}\delta_{e,f^{-1}}\\
			&=\left\{ \begin{array}{ll}
				\frac{1}{\sqrt{\deg u \deg v}} &\mbox{ if }
				uv\in E(G) \\ 
				0 &\textnormal{ otherwise.}
			\end{array}\right. \\
			&=\left(D^{-\frac{1}{2}}AD^{-\frac{1}{2}}\right)_{uv}.
		\end{align*} 
		Therefore $P=D^{-\frac{1}{2}}AD^{-\frac{1}{2}}$. Now $P(D^\frac{1}{2} \j)=(D^{-\frac{1}{2}}AD^{-\frac{1}{2}})(D^\frac{1}{2}\j)=D^\frac{1}{2}\j$. Thus $1$ is an eigenvalue of $P$ with corresponding eigenvector $D^\frac{1}{2}\j$.
		
		Let $\mu$ be an eigenvalue of $P$ with corresponding eigenvector $y$, that is, $Py=\mu y$. Let $y=(y_u)_{u\in V(G)}$. This gives
		$$|\mu|| y_u|\leq \sum_{v\in V(G)} P_{uv} |y_v|~~\text{for each}~u\in V(G).$$
		Define $x=D^\frac{1}{2}\textbf{1}$ and let $x=(x_u)_{u\in V(G)}$.	Thus we have
		\begin{align*}
			|\mu|\sum_{u\in V(G)} x_u|y_u|&\leq\sum_{u\in V(G)} x_u\sum_{v\in V(G)}P_{uv}|y_v|\\
			&=\sum_{v\in V(G)} |y_v|\sum_{u\in V(G)}P_{vu}x_u \\
			&=\sum_{v\in V(G)} x_v |y_v|.
		\end{align*}
		Note that $x_u>0$ for each $u\in V(G)$. Therefore $|\mu|\leq 1$.
	\end{proof}
	By Lemma~\ref{evc_P}, if  $G$ is a regular graph, then the discriminant of the graph can be expressed in terms of the adjacency matrix.
	\begin{lema}\label{reg}
		If $G$ is a $k$-regular graph, then $P=\frac{1}{k}A$. 
	\end{lema}
	In \cite[Theorem~6.5]{kubota2}, Kubota and Yoshino showed that a regular graph exhibits perfect state transfer  from a vertex $u$ to a vertex $v$ at time $\tau$ if and only if  $T_\tau(P)\eu=\ev$. In \cite{peak}, Guo and Schmeits mentioned that this holds for any graph, without requiring regularity. Here, we provide an explicit proof of this result.
	\begin{lema}\label{defpst}
		Let $u$ and $v$ be two vertices of a graph $G$. Then perfect state transfer occurs from $u$ to $v$ at time $\tau$ if and only if $T_\tau(P)\eu=\ev$.
	\end{lema}
	\begin{proof}
		The perfect state transfer occurs from $u$ to $v$ at time $\tau$ if and only if 
		\begin{align*}
			1&=\lvert\ip{U^\tau \Phi_u}{\Phi_v}\rvert \tag*{(by Lemma~\ref{st11})}\\
			&=\lvert\ip{U^\tau N^*\eu}{N^*\ev}\rvert  \\
			&=\lvert\ip{NU^\tau N^*\eu}{\ev}\rvert\\
			&=\lvert\ip{T_\tau(P)\eu}{\ev}\rvert.\tag*{(by Lemma \ref{ch11})}
		\end{align*}
		Note that $||T_\tau(P)\mathbf{e}_w||\leq 1$ for any vertex $w$. In fact, 
		\begin{align*}
			||T_\tau(P)\ew||^2&=\ip{T_\tau(P)\ew}{T_\tau(P)\ew}\\
			&=\ip{\sum_{r=0}^{m} T_\tau(\mu_r) E_r\ew}{\sum_{r=0}^{m} T_\tau(\mu_r) E_r\ew} \\
			&=\sum_{r=0}^{m}T_\tau(\mu_r)^2\ip{E_r\ew}{E_r\ew}\\
			&\leq \sum_{r=0}^{m}\ip{E_r\ew}{E_r\ew}=1.
		\end{align*}
		Since $T_\tau(P)$ is a real matrix and $||T_\tau(P)\mathbf{e}_u||\leq 1$, it follows that $\lvert\ip{T_\tau(P)\eu}{\ev}\rvert=1$ if and only if there exists $\gamma\in\{-1,1\}$ such that $T_\tau(P)\eu=\gamma\ev$.

		Next, we show that $\gamma$ must be equal to $1$. By \eqref{sd}, we have $T_\tau(P)=\sum_{j=0}^{d} T_\tau(\mu_j) E_j$ for any non-negative integer $\tau$.  From $T_\tau(P)\eu=\gamma\ev$  and the spectral decomposition of $P$, it follows that $T_\tau(\mu_0)E_0\eu=\gamma E_0\ev$. By Lemma~\ref{evc_P}, $\mu_0=1$, and so $E_0\eu=\gamma E_0\ev$. Since $1$ is an eigenvalue of $P$ with corresponding eigenvector $D^\frac{1}{2}\j$, we have $E_0 D^\frac{1}{2}\j=D^\frac{1}{2}\j$. Thus,
		$$\sum_{u\in V(G)}\sqrt{\deg u}=\ip{D^\frac{1}{2}\j}{\eu}=\ip{D^\frac{1}{2}\j}{E_0\eu}=\ip{D^\frac{1}{2}\j}{\gamma E_0\ev}=\gamma\ip{D^\frac{1}{2}\j}{\ev}=\gamma \sum_{u\in V(G)}\sqrt{\deg u}.$$
		Hence $\gamma=1$.
	\end{proof}
	The \emph{Schur product} of two matrices $M$  and $N$  of the same size is defined by $(M \circ N)_{a b} = M_{a b} \cdot N_{a b}.$ We denote  by $\mathbf{O}$ the all-zero matrix.
	\begin{theorem}\label{association1}
		Let $G$ be a graph belonging to an association scheme $\mathcal{A}$. Then perfect state transfer occurs from  vertex $u$ to another vertex $v$ on $G$ at time $\tau$ if and only if there exists $B\in\mathcal{A}$ such that $T_\tau(P)=B$, where $B$ is a permutation matrix of order $2$ with no fixed points and satisfies $B_{uv}=1$.
	\end{theorem}
	\begin{proof}
		Let $\mathcal{A}=\{A_0,\hdots,A_d\}$ and $A$ be the adjacency matrix of $G$. Since $T_\tau(P)$ is a polynomial of $A$ for each non-negative integer $\tau$, $T_\tau(P)$ belongs to the  Bose-Mesner algebra of the scheme. Thus we have
		\begin{equation}\label{asso}
			T_\tau(P)=\sum_{i=0}^{d}\alpha_iA_i,
		\end{equation}
		where $\alpha_i\in\Cl$. By Lemma~\ref{defpst}, if perfect state transfer occurs from vertex $u$ to vertex $v$ at time $\tau$, then $T_\tau(P)\eu=\ev$. Therefore $uv$-th entry of $T_\tau(P)$ is $1$. Note that $||T_\tau(P)\mathbf{e}_w||\leq 1$ for any vertex $w$. This implies that the $u$-th row $T_\tau(P)$ is $\ev^t$.

		Consider the class $A_\ell$ of $\mathcal{A}$ such that the $uv$-th entry of $A_\ell$ is $1$. Then for $i\neq \ell$, the $uv$-th entry of $A_i$ is $0$. Therefore comparison of $uv$-th entry of both sides of \eqref{asso} gives $\alpha_\ell=1$.   Now the facts that the $u$-th row $T_\tau(P)$ is $\ev^t$ and $\alpha_\ell=1$ altogether imply that the $u$-th row of $A_\ell$ is also $\ev^t$. Since any matrix in the Bose-Mesner algebra has constant row sum and column sum, $A_\ell$ is a permutation matrix.  For $0\leq i, j\leq d$, it follows from the definition of association scheme that
		\[ 	A_i \circ A_j =
		\begin{cases}
			A_i & \text{if } i = j \\
			0 & \text{otherwise}.
		\end{cases} 	\]
		Thus we have $T_\tau(P)\circ A_\ell=A_\ell$. This implies that if the $xy$-th entry of $A_\ell$ is $1$, then the $xy$-th entry of $T_\tau(P)$ is also $1$. Now using the fact $||T_\tau(P)\mathbf{e}_w||\leq 1$ for any vertex $w$, we  conclude that $x$-th row of $T_\tau(P)$ is $\mathbf{e}_y^t$. Thus $T_\tau(P)=A_\ell$.  Since $A_\ell$ is a symmetric permutation matrix and $A_\ell\neq A_0$, we conclude that $A_\ell$ is a permutation matrix of order $2$ with no fixed points.
		
		The converse  follows directly from Lemma~\ref{defpst}.		
	\end{proof}
	The following result gives the minimum time at which perfect state transfer can occur on graphs in an association scheme and also provides a useful necessary condition for perfect state transfer.
	\begin{lema}\label{period_ass}
		Let $G$ be a graph belonging to an association scheme. If $G$ exhibits perfect state transfer at minimum time $\tau$, then $G$ is $2\tau$-periodic.
	\end{lema}
	\begin{proof}
		Suppose $G$ exhibits perfect state transfer at minimum time $\tau$. Let $\mu\in\spec_P(G)$. By Theorem~\ref{association1}, we have $T_\tau(\mu)=\pm1$. Then by Lemma~\ref{ch}, $\mu=\cos \frac{m}{\tau}\pi$ for some positive integer $m$. Consequently,  Theorem~\ref{evu_grover} gives that $U^{2\tau}=I$.  Thus $G$ is periodic, and suppose it is $k$-periodic. Then $\tau\in\{1,\hdots,k-1\}$ and $k$ is the least positive integer such that $U^k=I$.  This implies that $k\divides2\tau$ and hence $k=2\tau$. 
	\end{proof}
	The following result is an easy consequence of Lemma~\ref{period_ass}.
	\begin{prop}
		Let $G$ be a graph belonging to an association scheme. If $G$ exhibits perfect state transfer from $u$ to $v$ as well as from $u$ to $w$, then $v=w$.
	\end{prop}
	\begin{proof}
		By Lemma~\ref{period_ass}, if $G$ exhibits perfect state transfer from $u$ to $v$ at time $\tau$, then perfect state transfer also occurs from $u$ to $w$ at time $\tau$.	Then by Lemma~\ref{defpst}, it follows that $\ev=T_\tau(P)\eu=\ew$. This completes the proof.
	\end{proof}
	Consider an association scheme $\mathcal{A}$ with $d$ classes that contains a permutation matrix $B$ of order $2$ with no fixed points. Let $\{E_0, \dots, E_d\}$ be the set of principal idempotents of the scheme $\mathcal{A}$. Since $B$ has eigenvalues $\pm1$, it follows that $B E_j = \pm E_j$ for $j \in \{0, \dots, d\}.$ Define a partition $(\mathcal{I}_B^+,\mathcal{I}_B^-)$ of $\{0,\hdots,d\}$ where $j\in\mathcal{I}_B^+$ if $BE_j=E_j$ and $j\in\mathcal{I}_B^-$ if $BE_j=-E_j$.
	\begin{theorem}\label{association2}
		Let $G$ be a graph belonging to an association scheme $\mathcal{A}$ with $d$ classes. Let $P$ be the discriminant of $G$ with distinct eigenvalues $\mu_0,\hdots,\mu_d$ such that $\mu_0>\cdots>\mu_d$. Then perfect state transfer occurs from $u$ to $v$ on $G$ at time $\tau$ if and only if the following two conditions are satisfied.
		\begin{enumerate}
			\item[(i)]  The scheme contains a class $B$ such that $B$ is a permutation matrix of order $2$ with no fixed points and $B_{uv}=1$.
			\item[(ii)] $T_\tau(\mu_j)=1$ for each $j\in \mathcal{I}_B^+$ and $T_\tau(\mu_j)=-1$ for each $j\in \mathcal{I}_B^-$.
		\end{enumerate}
	\end{theorem}
	\begin{proof}
		By Theorem~\ref{association1}, Condition (i) is necessary for the occurrence of perfect state transfer from $u$ to $v$ on $G$ at time $\tau$. Suppose Condition (i) holds. It is enough to prove that perfect state transfer occurs from $u$ to $v$ on $G$ at time $\tau$ if and only if Condition (ii) holds. Note that 
		$$T_\tau(P)=\sum_{j\in\mathcal{I}_B^+}T_\tau(\mu_j)E_j+\sum_{j\in\mathcal{I}_B^-}T_\tau(\mu_j)E_j$$
		and
		$$B=\sum_{j\in\mathcal{I}_B^+}E_j+\sum_{j\in\mathcal{I}_B^-}(-1)E_j.$$
		Since $T_\tau(P)$ and $B$ are elements of $\Cl[\mathcal{A}]$ and $\{E_0,\hdots,E_d\}$ forms a basis of $\Cl[\mathcal{A}]$, we have $T_\tau(P)=B$ if and only if  $T_\tau(\mu_j)=1$ for each $j\in \mathcal{I}_B^+$ and $T_\tau(\mu_j)=-1$ for each $j\in \mathcal{I}_B^-$. Therefore, the result follows from Theorem~\ref{association1}.
	\end{proof}
	We now apply the preceding results to two important families of association schemes, namely the Hamming scheme and the Johnson scheme.
	\subsection{Hamming Scheme}
	Let $d$ and $q$ be two positive integers, and $F$ be a set of $q$ elements. For $i\in\{1,\hdots, d\}$, define the graph $H(d,q,i)$ with vertex set $F^d$, and two vertices $x:=(x_1,\hdots,x_d)$ and $y:=(y_1,\hdots,y_d)$ are adjacent if the number of positions in which they differ is exactly $i$, that is, $|\{j:x_j\neq y_j,1\leq j\leq d\}|=i$. Let $A_0=I$ and for $i\in\{1,\hdots,d\}$, let $A_i$ be the adjacency matrix of the graph $H(d,q,i)$. It is easy to verify that the set $\mathcal{A}:=\{A_0,\hdots, A_d\}$ is an association scheme with $d$ classes. This scheme, denoted $H(d,q)$, is known as the \emph{Hamming scheme}. Observe that the graph $H(d,2,1)$ is the well-known $d$-dimensional hypercube.

	For $i\in\{1,\hdots,d\}$, the graph $H(d,q,i)$ is a regular graph of order $q^d$  and valency $(q-1)^i\binom{d}{i}$. The next result provides a necessary condition for the occurrence of perfect state transfer on graphs in the Hamming scheme.
	\begin{lema}\label{q=2}
		Let $G$ be a graph belonging to the Hamming scheme $H(d,q)$. If $G$ exhibits perfect state transfer, then $q=2$. Moreover, in $H(d,2)$, the class $A_d$ is a permutation matrix of order $2$ with no fixed points.
	\end{lema}
	\begin{proof}
		If perfect state transfer occurs on $G$, then by Theorem~\ref{association1}, the scheme contains a permutation matrix $B$ of order 2 with no fixed points. Hence, there exists an index $i \in \{1, \ldots, d\}$ such that $A_i = B$. This implies that the graph $H(d,q,i)$ is regular of degree 1, and therefore
		$$(q-1)^i\binom{d}{i} = 1.$$
		This equation yields $(q-1)^i = 1 = \binom{d}{i}$. Since $i\geq 1$, we find that $q=2$ and $i=d$.
	\end{proof}
	For a non-negative integer $k$, $\binom{x}{k}$ is the polynomial defined by 
	$$\binom{x}{k}= \left\{ \begin{array}{ll}
		\frac{x(x-1)\cdots (x-k+1)}{k!} &\mbox{ if }
		k\neq 0 \\ 
		1 & \mbox{ if } k=0.
	\end{array}\right.$$ 
	For positive integers $d$, $q$, and $k\in\{0,\hdots,d\}$,  the Krawtchouk polynomial (see \cite[p. 209]{algcomb}) is defined by
	$$K_k(x)=K_k(x;d,q)=\sum_{\ell=0}^{k}(-1)^\ell (q-1)^{k-\ell} \binom{x}{\ell}\binom{d-x}{k-\ell}.$$
	For $i\in\{0,\hdots,d\}$, the eigenvalues (see \cite[Theorem~2.3]{algcomb}) of $A_i$ are given by 
	\begin{equation}\label{evcham}
		\lambda_j^i=K_i(j;d,q)~\text{for $j\in\{0,\hdots,d\}$}.
	\end{equation}
	We now provide a necessary and sufficient condition for the occurrence of perfect state transfer on a graph in the Hamming scheme $H(d,2)$.
	\begin{theorem}\label{pst_hamming}
		Let $G$ be a graph belonging to the Hamming scheme $H(d,2)$.  Let the distinct eigenvalues of the discriminant $P$ of $G$ be $\mu_0,\hdots,\mu_d$ such that $\mu_0>\cdots>\mu_d$. Then $G$ exhibits perfect state transfer at time $\tau$ if and only if 
		$T_\tau(\mu_j)=1$ if $j$ is even and $T_\tau(\mu_j)=-1$ if $j$ is odd.
	\end{theorem}
	\begin{proof}
		Let $\{E_0,\hdots,E_d\}$ be the set of all principal idempotents of the scheme $H(d,2)$. By Lemma~\ref{q=2}, the class $A_d$ is a permutation matrix of order $2$ with no fixed points. From Theorem~\ref{sd3}, it follows that $A_dE_j=\lambda_j^dE_j$ for $j\in\{0,\hdots,d\}$. From~\eqref{evcham}, we have
		$$\lambda_j^d=\sum_{\ell=0}^{d}(-1)^\ell \binom{j}{\ell}\binom{d-j}{d-\ell}.$$
		Note that $\binom{j}{\ell}\binom{d-j}{d-\ell}\neq 0$ only for $\ell=j$. Thus  $\lambda_j^d=(-1)^j$ for $j\in\{0,\hdots,d\}$. Consequently, $A_dE_j=(-1)^jE_j$ for $j\in\{0,\hdots,d\}$. This implies that $j\in \mathcal{I}_{A_d}^+$ if $j$ is even and $j\in \mathcal{I}_{A_d}^-$ if $j$ is odd. Therefore by Theorem~\ref{association2}, the result follows.
	\end{proof}
	
	We use an alternative expression for the Krawtchouk polynomials to prove our next results. From \cite[Chapter 5, Theorem 15]{correcting}, we have
	$$K_k(x;d,2)=\sum_{\ell=0}^{k}(-2)^\ell\binom{d-\ell}{k-\ell}\binom{x}{\ell}~\text{for $k\in\{0,\hdots,d\}$}.$$
	Note that $H(d,2,i)$ is a $\binom{d}{i}$-regular graph for $i\in\{1,\hdots,d\}$. Thus for $i\in\{1,\hdots,d\}$, the eigenvalues of the discriminant $P$ of $H(d,2,i)$ are given by 
	\begin{equation*}\label{ev_hamming}
		\mu_j^i=\frac{1}{\binom{d}{i}}\sum_{\ell=0}^{i}(-2)^\ell\binom{d-\ell}{i-\ell}\binom{j}{\ell}~\text{for $j\in\{0,\hdots,d\}$}.
	\end{equation*}
	In some special cases, we calculate the previous expression explicitly as follows. 
	\begin{align}
		\mu_0^i&=1 \label{mufirst1}\\
		\mu_1^i&=1-\frac{2i}{d}\label{mufirst}\\
		\mu_2^i&=1-\frac{4i}{d}+\frac{4i(i-1)}{d(d-1)}\\
		\mu_3^i&=1-\frac{6i}{d}+\frac{12i(i-1)}{d(d-1)}-\frac{8i(i-1)(i-2)}{d(d-1)(d-2)}\\
		\mu_4^i&=1-\frac{8i}{d}+\frac{24i(i-1)}{d(d-1)}-\frac{32i(i-1)(i-2)}{d(d-1)(d-2)}+\frac{16i(i-1)(i-2)(i-3)}{d(d-1)(d-2)(d-3)} \label{mulast}
	\end{align}
	The following theorem characterizes the occurrence of perfect state transfer on the classes of the Hamming scheme.
	\begin{theorem}\label{class_hamming}
		The class $H(d,q,i)$ of the Hamming scheme $H(d,q)$ exhibits perfect state transfer if and only if $(d,q,i)\in\left\{(d,2,d),(2,2,1)\right\}$.
	\end{theorem}
	\begin{proof}
		If $H(d, q, i)$ exhibits perfect state transfer, then by Lemma~\ref{q=2}, $q=2$. Note that $H(d,2,d)$ is a $1$-regular graph, and therefore $A(H(d, 2, d))=P(A(H(d, 2, d)))$. As $T_1(x)=x$, we observe that
		$$A(H(d, 2, d)) = A_d = T_1(P(H(d, 2, d))).$$ 
		Therefore by Theorem~\ref{association1}, the graph $H(d, 2, d)$ exhibits perfect state transfer at time $1$.  
		
		Now consider $i\in\{1, \ldots, d-1 \}$. By Lemma~\ref{period_ass}, if $H(d, q, i)$ exhibits perfect state transfer, then it must be periodic. Applying Theorem~\ref{ls} together with the Equation~\eqref{mufirst}, we find that
		$$\mu_1^i = 1 - \frac{2i}{d} \in \left\{ \pm 1, \pm \frac{1}{2}, 0 \right\},$$
		which implies that $i\in\left\{\frac{d}{2}, \frac{d}{4}, \frac{3d}{4} \right\}$. Additionally, we have
		$$\mu_2^i = 1 - \frac{4i}{d} + \frac{4i(i - 1)}{d(d - 1)} \in \left\{ \pm 1, \pm \frac{1}{2}, 0 \right\}.$$
		Since $i\in\{\frac{d}{2},\frac{d}{4},\frac{3d}{4}\}$ and $\mu_2^i\in\left\{\pm 1, \pm\frac{1}{2},0\right\}$, we conclude that the only possible pairs $(d,i)$ are $(2, 1),~ (4, 1)$ and $(4, 3)$.
		
		Let $(d,i)\in\{(4,1),(4,3)\}$. From equations \eqref{mufirst1} to \eqref{mulast}, we have $\spec_P(H(d,2,i))=\{\pm 1, \pm\frac{1}{2},0\}$. Then by Theorem~\ref{evu_grover} and Lemma~\ref{periodic}, the graph is $12$ periodic. Now, if $H(d,2,i)$ exhibits perfect state transfer at time $\tau$, then by Lemma~\ref{period_ass}, we must have $\tau=6$. However,  $T_6(\frac{1}{2})=1$. Therefore, Theorem~\ref{pst_hamming} implies that the graph $H(d,2,i)$ does not exhibit perfect state transfer.
		
		Now consider $(d,i)=(2,1)$. In this case, the graph $H(d,2,i)$ has distinct eigenvalues $1$, $0$ and $-1$. Note that $T_2(1)=1$, $T_2(0)=-1$ and $T_2(-1)=1$. Therefore by Theorem~\ref{pst_hamming}, the graph $H(d,2,i)$ exhibits perfect state transfer at time $2$. Thus, the desired result follows.
	\end{proof}
	The following result directly follows from the previous theorem.
	\begin{corollary}
		The hypercube $Q_d$ exhibits perfect state transfer if and only if $d\in\{1,2\}$.
	\end{corollary}
	\subsection{Johnson Scheme}
	Let $n$ and $k$ be two positive integers with $n\geq 2k$. For each $i\in\{0,\hdots,k-1\}$, the \emph{generalized Johnson graph}, denoted $J(n,k,i)$, is the graph whose vertex set is the set of all $k$-subsets of $\{1,\hdots,n\}$ and two vertices $E$ and  $F$ are adjacent if  $|E\cap F|=i$. Let $A_0=I$ and $A_i$ be the adjacency matrix of the graph $J(n,k,k-i)$ for $i\in\{1,\hdots, k\}$. It can be verified that the set $\mathcal{A}:=\{A_0,\hdots,A_k\}$ is an association scheme. This scheme is known as the \emph{Johnson scheme} and it is denoted by $J(n,k)$.                                                         
	
	The generalized Johnson graph $J(n,k,i)$ is a regular graph of order $\binom{n}{k}$ and valency $\binom{k}{i}\binom{n-k}{k-i}$ for $i\in\{0,\hdots, k-1\}$. For $i\in\{0,\hdots, k-1\}$, the eigenvalues (see \cite[page 220]{algcomb}) of the adjacency matrix of $J(n,k,i)$ are given by 
	\begin{equation}\label{evcj}
		\lambda_j^i=\sum_{\ell=0}^{k-i}(-1)^\ell\binom{j}{\ell}\binom{k-j}{k-i-\ell}\binom{n-k-j}{k-i-\ell},\quad j\in\{0,\hdots,k\}.
	\end{equation}
	Therefore the eigenvalues of the scheme $J(n,k)$ are given by $\lambda_j^{k-i}$. In other words, for each $i\in\{1,\hdots,k\}$, $A_i$ has eigenvalues $\lambda_j^{k-i}$ for $j\in\{0,\hdots,k\}$. The next result gives a necessary condition for the occurrence of perfect state transfer on the Johnson scheme $J(n,k)$.
	\begin{lema}\label{n=2k}
		Let $G$ be a graph belonging to the Johnson scheme $J(n,k)$. If $G$ exhibits perfect state transfer, then $n=2k$. Moreover, in $J(2k,k)$, the class $A_k$ is a permutation matrix of order $2$ with no fixed points.
	\end{lema}
	\begin{proof}
		If perfect state transfer occurs on $G$, then by Theorem~\ref{association1}, the Johnson scheme contains a permutation matrix $B$ of order 2 with no fixed points. Consequently, there exists an index $ i \in \{1, \ldots, k\} $ such that $B=A_i $. This means that the graph $J(n, k, k-i)$ is regular of degree 1, and therefore
		$$\binom{k}{k - i} \binom{n - k}{i} = 1.$$ 
		From this, we deduce that $i=k$ and $n=2k$.
	\end{proof}
	We now provide a necessary and sufficient condition for the occurrence of perfect state transfer of a graph in the Johnson scheme $J(2k,k)$.
	\begin{theorem}\label{pst_johnson}
		Let $G$ be a graph belonging to the Johnson scheme $J(2k,k)$. Let the distinct eigenvalues of the discriminant $P$ of $G$ be $\mu_0,\hdots,\mu_d$ such that $\mu_0>\cdots>\mu_d$. Then $G$ exhibits perfect state transfer at time $\tau$ if and only if $T_\tau(\mu_j)=1$ if $j$ is even and $T_\tau(\mu_j)=-1$ if $j$ is odd.
	\end{theorem}
	\begin{proof}
		By Lemma~\ref{n=2k}, $A_k$ is a permutation matrix of order $2$ with no fixed points in the classes of $J(2k,k)$. Hence Condition (i) of Theorem~\ref{association2} is satisfied. Let $E_0,\hdots,E_k$ be the principal idempotents of the scheme $J(2k,k)$. Recall that $(\mathcal{I}_{A_k}^+,\mathcal{I}_{A_k}^-)$  is the partition of the index set $\{0,\hdots,k\}$ such that $j\in\mathcal{I}_{A_k}^+$ if $A_kE_j=E_j$ and $j\in\mathcal{I}_{A_k}^-$ otherwise.   By Theorem~\ref{sd3}, it follows that $A_kE_j=\lambda_j^0E_j$ for $j\in\{0,\hdots,k\}$. From~\eqref{evcj}, we find 
		$$\lambda_j^0=(-1)^j.$$
		Therefore $A_kE_j=(-1)^jE_j$ for $j\in\{0,\hdots,d\}$. This implies that $j\in \mathcal{I}_{A_k}^+$ if $j$ is even and $j\in \mathcal{I}_{A_k}^-$ if $j$ is odd. By Theorem~\ref{association2} with this partition, we have the desired result.
	\end{proof}
	Since $J(2k,k,i)$ is a $\binom{k}{i}^2$-regular graph, by \eqref{evcj} the eigenvalues of the discriminant of $J(2k,k,i)$ are given by 
	\begin{equation*}
		\mu_j^i=\frac{1}{\binom{k}{i}^2}\sum_{\ell=0}^{k-i} (-1)^\ell\binom{j}{\ell}\binom{k-j}{k-i-\ell}^2~\text{for $j\in\{0,\dots,k\}$}.
	\end{equation*}
	\begin{theorem}\label{class_johnson}
		The class $J(n,k,i)$ of the Johnson scheme $J(n,k)$ exhibits perfect state transfer if and only if $(n,k,i)\in\left\{(2k,k,0),(4,2,1)\right\}$.
	\end{theorem}
	\begin{proof}
		If $J(n,k,i)$ exhibits perfect state transfer, then by Lemma~\ref{n=2k}, we must have $n=2k$.  For $i=0$, we observe that $A(J(2k,k,0))=P(J(2k,k,0))$, and hence
		$$A(J(2k,k,0))=A_k=T_1(P(J(2k,k,0))).$$
		Therefore by Theorem~\ref{association1}, $J(2k,k,0)$ exhibits perfect state transfer at time $1$. Now consider the case $i\in\{1,\hdots, k-1\}$. By Lemma~\ref{period_ass}, if $J(2k,k,i)$ exhibits perfect state transfer at time $\tau$, then the graph is $2\tau$-periodic. Therefore by Theorem~\ref{ls}, $\mu_j^i\in\{\pm1,\pm\frac{1}{2},0\}$ for $j\in\{0,\dots,k\}$. In particular,
		$$\mu_k^i=\frac{(-1)^{k-i}}{\binom{k}{i}}\in \left\{\pm1,\pm\frac{1}{2},0\right\}.$$
		This implies that $\binom{k}{i}\in\{1,2\}$. Since $i\in\{1,\hdots, k-1\}$, we have $\binom{k}{i}\neq 1$. Hence we must have $\binom{k}{i}=2$, giving $(k,i)=(2,1)$. Now, the distinct eigenvalues of $P(J(4,2,1))$ are $1,0$ and $-\frac{1}{2}$. Thus the graph $J(4,2,1)$ is $12$-periodic, and therefore $\tau=6$. We find $T_6(1)=1$, $T_6(0)=-1$ and $T_6(-\frac{1}{2})=1$. Therefore by Theorem~\ref{pst_johnson}, the graph $J(4,2,1)$ exhibits perfect state transfer at time $6$.
	\end{proof}
	\section{Perfect state transfer on distance-regular graphs}
	Distance-regular graphs form an important class of graphs in algebraic combinatorics, known for their highly symmetric and regular distance properties. For a distance-regular graph $G$ with diameter $d$, recall the association scheme $\{A_0,\hdots,A_d\}$ from Subsection~\ref{def_distance}. 
	\begin{lema}[{\cite[Proposition~11.6.2]{spectraofgraphs}}]\label{distance_partition}
		Let $G$ be a distance-regular graph of diameter $d$. Let the distinct eigenvalues of the adjacency matrix $A$ of $G$ be $\ld_0,\hdots,\ld_d$ such that $\ld_0>\cdots>\ld_d$ with corresponding idempotents $E_0,\hdots,E_d$. If $G$ is antipodal with fibres of size two, then $$A_dE_j=(-1)^jE_j~\text{for each $j\in\{0,\hdots,d\}$}.$$
	\end{lema}
	The following result provides a simple necessary and sufficient condition for the occurrence of perfect state transfer on distance-regular graphs. A distance-regular graph of diameter $d$ has exactly $d+1$ distinct eigenvalues of its adjacency matrix, and thus it will also have  $d+1$ distinct discriminant eigenvalues.
	\begin{theorem}\label{pst_distance}
		Let $G$ be a distance-regular graph of diameter $d$. Let the distinct eigenvalues of the discriminant $P$ of $G$ be $\mu_0,\hdots,\mu_d$ such that $\mu_0>\cdots>\mu_d$. Then perfect state transfer occurs on $G$ at time $\tau$ if and only if the following two conditions hold.
		\begin{enumerate}[label=(\roman*)]	
			\item\label{pst_distance_i}  $G$ is antipodal with fibres of size $2$, and
			\item $T_\tau(\mu_j)=1$ if $j$ is even and $T_\tau(\mu_j)=-1$ if $j$ is odd.
		\end{enumerate}
		Moreover, if (i) and (ii) hold, then $G$ exhibits perfect state transfer between antipodal vertices.
	\end{theorem}
	\begin{proof}	
		Let $G$ exhibit perfect state transfer at time $\tau$. Then by Theorem~\ref{association2}, $A_i$ is a permutation matrix of order $2$ with no fixed points for some $i\in\{1,\hdots,d\}$. Suppose, if possible $i<d$. Then we have $d>1$. Let the $uv$-th entry of $A_i$ be $1$. Then the vertex $v$ is the only vertex at distance $i$ from $u$, giving $b_{i-1} = 1$. By Lemma~\ref{intarray}, it follows that $b_j = 1$ for all $j \geq i$. In particular, there is a unique vertex at distance $d$ from $u$, say $x$, with $\deg x = 1$. Since $G$ is connected and regular, $\deg x=1$ forces the graph to be $K_2$, contradicting the fact that $d > 1$. Hence we must have $i=d$. From the fact that $A_d$ is a permutation matrix of order $2$ with no fixed points, we conclude that $G$ is antipodal with fibres of size $2$. Now Condition (ii) follows from this fact along with Theorem~\ref{association2} and Lemma~\ref{distance_partition}.
		
		Conversely, if Conditions (i) and (ii) hold, then Theorem~\ref{association2} and Lemma~\ref{distance_partition} give that $G$ exhibits perfect state transfer at time $\tau$ between antipodal vertices.
	\end{proof}
	The cycle $C_n$ and the complete graph $K_n$ are both distance-regular graphs. The intersection array of $C_n$ is $\{2,1,\hdots,1;1,1,\hdots,1\}$ if $n$ is odd, and $\{2,1,\hdots,1;1,\hdots,1,2\}$ if $n$ is even. The intersection array of $K_n$ is $\{n-1,1\}$ if $n>1$. In the next two results, we use Theorem~\ref{pst_distance} to completely characterize the existence of perfect state transfer on cycles and complete graphs.
	\begin{lema}\label{pst_cycle}
		The cycle $C_n$ exhibits perfect state transfer if and only if $n$ is even. Moreover, if $n$ is even, then $C_n$ exhibits perfect state transfer between antipodal vertices at time $\frac{n}{2}$.
	\end{lema}
	\begin{proof}
		If $C_n$ exhibits perfect state transfer, then by Condition~\ref{pst_distance_i} of Theorem~\ref{pst_distance}, $n$ must be even. Then $C_n$ is an antipodal distance-regular graph with fibres of size $2$.  Note that $C_n$ is $n$-periodic. Thus if $C_n$ exhibits perfect state transfer at time $\tau$, then $\tau=\frac{n}{2}$. The distinct discriminant eigenvalues of $C_n$ are $\mu_j=\cos \frac{2\pi j}{n}$ for $j\in\{0,\hdots,\frac{n}{2}-1\}$ with $\mu_0>\cdots>\mu_{\frac{n}{2}-1}$. Then by \eqref{chb}, $T_\frac{n}{2}(\mu_j)=1$ if $j$ is even and  $T_\frac{n}{2}(\mu_j)=-1$ if $j$ is odd. Therefore by Theorem~\ref{pst_distance}, $C_n$ exhibits perfect state transfer at time $\frac{n}{2}$.
	\end{proof}
	\begin{lema}\label{pst_complete}
		The complete graph $K_n$ exhibits perfect state transfer if and only if $n=2$.
	\end{lema}
	\begin{proof}
		The complete graph $K_n$ is an antipodal distance-regular graph with fibres of size $n$. Therefore by Theorem~\ref{pst_distance}, the result follows.
	\end{proof}
	The only distance-regular graphs of diameter $1$ are the complete graphs. Consequently, the only distance-regular graph of diameter $1$ exhibiting perfect state transfer is $K_2$. In the next two theorems, we characterize perfect state transfer on distance-regular graphs of diameter $2$ and diameter $3$. Let $G$ be a regular graph that is neither complete nor empty. Then $G$ is called \emph{strongly regular} with \emph{parameters} $(n,k,a,c)$ if it is a $k$-regular graph on $n$ vertices such that every pair of adjacent vertices has exactly $a$ common neighbors, and every pair of non-adjacent vertices has exactly $c$ common neighbors.
	\begin{theorem}\label{pst_distance2}
		Let $G$ be a distance-regular graph of diameter $2$. Then $G$ exhibits perfect state transfer if and only if $G$ is isomorphic to the cycle $C_4$ or complete tripartite graph $K_{2,2,2}$.
	\end{theorem}
	\begin{proof}
		It is well known that a graph is distance-regular with diameter $2$ if and only if it is strongly-regular. Let $G$ be a strongly regular graph with parameters $(n,k,a,c)$. If $G$ exhibits perfect state transfer, then $G$ is antipodal. We claim that $c=k$.
		
		Suppose $c\neq k$. Let $u$ and $v$ be two non-adjacent vertices of $G$. Then there exists $x\in V(G)$ such that $x$ is adjacent to $u$, but $x$ is not adjacent to $v$. Thus $\dist_G(u,v)=2$ and $\dist_G(x,v)=2$. Since $G$ is antipodal, it must also be that $\dist_G(u,x)=2$, which contradicts that $x$ is adjacent to $u$. Thus, we have $c=k$.
		
		Note that the parameters $(n,k,a,c)$ of a strongly regular graph satisfy $(n-k-1)c=k(k-a-1)$. Substituting $c=k$ into this equation, we find $a=2k-n$. Therefore $G$ is a strongly-regular graph with parameters $(n,k,2k-n,k)$, which characterizes $G$ as a complete $(n-k)$-partite graph (for instance, see~\cite[Proposition~10]{greaves}). This graph is antipodal with fibres of size $n-k$. 
		
		Thus $G$ exhibits perfect state transfer only if $n-k=2$. The distinct discriminant eigenvalues of such strongly-regular graphs are $1,0$ and $-\frac{2}{n-2}$. By Lemma~\ref{period_ass} and Theorem~\ref{ls}, it follows that  $ -\frac{2}{n-2}\in\{\pm 1,\pm \frac{1}{2},0\}$, which implies $n\in\{4,6\}$. For $n=4$, the cycle $C_4$ exhibits perfect state transfer by Lemma~\ref{pst_cycle}. For $n=6$, the distinct discriminant eigenvalues are $1,-\frac{1}{2},0$. Now, $T_6(1)=1$, $T_6(0)=-1$ and $T_6(-\frac{1}{2})=1$. Thus by theorem~\ref{pst_distance}, $K_{2,2,2}$ exhibits perfect state transfer at time $6$.
	\end{proof}
	The well-known Petersen graph is a distance-regular graph of diameter $2$. Therefore by the previous result, the Petersen graph does not exhibit perfect state transfer. The previous theorem also provides a complete characterization of perfect state transfer on strongly regular graphs and regular complete multipartite graphs. In their main result (see \cite[Theorem 1.2]{kubota1}), Kubota and Segawa completely determined which regular complete multipartite graphs exhibit perfect state transfer. However, our proof of the characterization is significantly simpler than the one presented in \cite{kubota1}.
	
	We now consider distance-regular graphs of diameter $3$. Let $G$ be a graph, and suppose $\pi$ is a partition of $V(G)$ into cells satisfying the following two conditions:
	\begin{enumerate}
		\item[(i)] each cell is an independent set, and 
		\item[(ii)] the subgraph induced by two distinct cells is either empty or $1$-regular.
	\end{enumerate}
	Let $G/\pi$ be a graph whose vertex set is $\pi$ and two vertices $X$ and $Y$ are adjacent if  the subgraph induced by $X\cup Y$ is $1$-regular. In this case, we say that $G$ is a \emph{covering} of $G/\pi$. Note that if $G/\pi$ is connected, then all cells in $\pi$ have the same size.  The common size is referred to as \emph{index} of the covering, and when this holds, we say that $G$ is an \emph{$r$-fold covering} of $G/\pi$. For example, the cycle $C_6$ is a $2$-fold covering of the complete graph $K_3$. We refer to \cite{cover} for more details about the covers of a graph.
	
	The intersection array (see \cite[Lemma~3.1]{cover}) of a antipodal distance-regular $r$-fold covering of $K_n$ is
	$$\{n-1,(r-1)c,1;1,c,n-1\},$$
	where $c$ is the number of common neighbours of two non-adjacent vertices of two different fibres. Therefore the set of distinct eigenvalues (see \cite[Equation (1)]{cover}) of the adjacency matrix of such graphs is 
	\begin{equation}\label{evc_cover}
		\left\{n-1,~ \frac{\delta+\sqrt{\Delta}}{2},~-1~, \frac{\delta-\sqrt{\Delta}}{2}\right\},
	\end{equation}
	where $\delta=n-rc-2$ and $\Delta=\delta^2+4(n-1)$. Note that these graphs have the valency $n-1$.
	\begin{theorem}\label{pst_distance3}
		Let $G$ be a distance-regular graph of diameter $3$. Then $G$ exhibits perfect state transfer if and only if $G$ is isomorphic to the cycle $C_6$.
	\end{theorem}
	\begin{proof}
		It is well known (for instance, see \cite[Section~4.2B]{brouwer_drt}) that an antipodal distance-regular graph of diameter $3$ is an $r$-fold covering of the complete graph $K_n$ for some positive integers $r$ and $n$. This implies that the fibres of $G$ are of size $r$. If $G$ exhibits perfect state transfer, then by Theorem~\ref{pst_distance}, we have $r=2$. Now from  \eqref{evc_cover}, it follows that $-\frac{1}{n-1}\in\spec_P(G)$. Then, by Lemma~\ref{period_ass} and Theorem~\ref{ls}, we must have $-\frac{1}{n-1}\in\{\pm1,\pm\frac{1}{2},0\}$, which implies $n=2$ or $n=3$.
		
		However, a $2$-fold covering of $K_2$ is the disjoint union $K_2\cup K_2$, which is not distance-regular. For $n=3$, the distance-regular $2$-fold covering of $K_3$ is isomorphic to the cycle $C_6$. Hence by Lemma~\ref{pst_cycle}, the result follows. 
	\end{proof}
	We now investigate perfect state transfer on integral distance-regular graphs. As previously noted, periodicity is a necessary condition for perfect state transfer on distance-regular graphs. In \cite{bhakta1}, we presented a characterization of integral periodic regular graphs.
	\begin{theorem}[{\cite[Theorem~5.9]{bhakta1}}]\label{int_periodic}
		A graph $G$ is regular, integral and periodic if and only if it is either the cycle $C_6$ or the complete bipartite graph $K_{k,k}$ or the complete tripartite graph $K_{k,k,k}$ or $\spec_P(G)=\{1, \pm \frac{1}{2},0\}$ or $\spec_P(G)=\{\pm 1,\pm \frac{1}{2},0\}$.		
	\end{theorem}
	We apply Theorem~\ref{int_periodic} to show that, up to isomorphism, only four graphs in the class of integral distance-regular graphs exhibit perfect state transfer.
	\begin{theorem}\label{pst_int_distance}
		Let $G$ be an integral distance-regular graph. Then $G$ exhibits perfect state transfer if and only if $G$ is isomorphic to one of $K_2$, $C_4$, $C_6$ and the complete tripartite graph $K_{2,2,2}$.
	\end{theorem}
	\begin{proof}
		By Theorem~\ref{pst_distance2} and Lemmas~\ref{pst_cycle} and~\ref{pst_complete}, the graphs $K_2$, $C_4$, $C_6$, and $K_{2,2,2}$ exhibit perfect state transfer. Suppose $G$ is an integral distance-regular graph exhibiting perfect state transfer at time $\tau$. Then Lemma~\ref{period_ass} and Theorem~\ref{int_periodic} imply that $G\in\{C_6,K_{k,k},K_{k,k,k}\}$ for some positive integer $k$ or $\spec_P(G) = \{1, \pm \frac{1}{2}, 0\}$ or $\spec_P(G) = \{\pm 1, \pm \frac{1}{2}, 0\}$.  For $k = 1$, we have $K_{k,k} = K_2$ and $K_{k,k,k} = K_3$. However, $K_3$ does not exhibit perfect state transfer. For $k \geq 2$, the proof of Theorem~\ref{pst_distance2} shows that the graphs $K_{k,k}$ and $K_{k,k,k}$ exhibit perfect state transfer only for $k=2$, that is, the graph is either $C_4$ or $K_{2,2,2}$.
		
		Now suppose $\spec_P(G)=\{1, \pm \frac{1}{2},0\}$. By Theorem~\ref{evu_grover} and Lemma~\ref{periodic}, the period of $G$ is $12$. Then by Lemma~\ref{period_ass}, $\tau=6$. Observe that $1>\frac{1}{2}>0>-\frac{1}{2}$ and $T_6(\frac{1}{2})=1$. Therefore by Theorem~\ref{pst_distance}, $G$ does not exhibit perfect state transfer. Similarly, $G$ does not exhibit perfect state transfer for the case $\spec_P(G)=\{\pm 1,\pm \frac{1}{2},0\}$ as well.
	\end{proof}
	\section*{Acknowledgements}
	The first author gratefully acknowledges the support of the Prime Minister's Research Fellowship (PMRF), Government of India (PMRF-ID: 1903298).

\end{document}